\documentclass{amsart}
\usepackage{graphicx}
\usepackage{amssymb}
\usepackage{cite}
\usepackage[top=3cm,left=3.3cm,right=3.3cm,bottom=3cm]{geometry}
\newtheorem{cor}{Corollary}[section]
\newtheorem{prop}{Proposition}[section]

\newcommand{\sgn}{\operatorname{sgn}}

\newcommand{\rX}{\mathrm{X}}
\newcommand{\tz}{\tilde{z}}

\newcommand{\hA}{{\hat{A}}}
\newcommand{\al}{{\alpha}}
\newcommand{\be}{{\beta}}
\newcommand{\hB}{{\hat{B}}}
\newcommand{\hP}{{\hat{P}}}

\newcommand{\hW}{{\hat{W}}}
\newcommand{\rL}{\mathrm{L}}

\newcommand{\Lag}[2]{L_{#2}^{{(#1)}}}      
\newcommand{\XLagI}[3]{L_{#2,#3}^{\I(#1)}}  
\newcommand{\XLagII}[3]{L_{#2,#3}^{{\II}(#1)}}                             %

\newcommand{\I}{{\rm{I}}}
\newcommand{\II}{\rm{II}}
\newcommand{\ezeta}[3]{\zeta_{#2,#3}^{#1}}
\newcommand{\Jac}[3]{P_{#3}^{(#1,#2)}}
\newcommand{\Jacp}[3]{{P_{#3}^{(#1,#2)}}'}
\newcommand{\XJac}[4]{{\hat P}_{#3,#4}^{(#1,#2)}}

\newcommand{\cL}{\mathcal{L}}
\newcommand{\Wr}{\operatorname{Wr}}
\newcommand{\cU}{\mathcal{U}}
\newcommand{\cP}{\mathcal{P}}

\newcommand{\Cset}{\mathbb{C}}
\newcommand{\Pud}[2]{{P^{(#1)}_{#2}}}
\newcommand{\hPud}[3]{{{\hat P}^{(#1)}_{#2,#3}}}
\newcommand{\hPudp}[3]{{{\hat P}_{#2,#3}^{(#1)'}}}

\begin{document}

\title[]{Asymptotic behaviour of zeros of exceptional Jacobi and Laguerre
polynomials}
\author{David G\'omez-Ullate}
\address{ Departamento de F\'isica Te\'orica II, Universidad Complutense de
Madrid, 28040 Madrid, Spain.}
\author{ Francisco Marcell\'an }
\address{Departamento de Matem\'aticas, Universidad Carlos III de Madrid, 28911
Legan\'es, Spain.}
\author{Robert Milson}
\address{Department of Mathematics and Statistics, Dalhousie University,
Halifax, NS, B3H 3J5, Canada.}
\begin{abstract}
  The location and asymptotic behaviour for large $n$ of the zeros of
exceptional Jacobi and Laguerre polynomials are discussed.
The zeros of exceptional polynomials fall into two classes: the \textit{regular
zeros}, which lie in the interval of orthogonality and the \textit{exceptional
zeros}, which lie outside that interval.
We show that the regular zeros have two interlacing properties: one is the
natural interlacing between consecutive polynomials as a consequence of their
Sturm-Liouville character, while the other one shows interlacing between the
zeros of exceptional and classical polynomials. A generalization of the
classical Heine-Mehler formula is provided for the exceptional polynomials,
which allows to derive the asymptotic behaviour of their regular zeros.
We also describe the location and the asymptotic behaviour of the
\textit{exceptional zeros}, which converge
for large $n$ to fixed values.
\end{abstract}
\keywords{Exceptional orthogonal polynomials, zeros, outer relative
asymptotics, Mehler-Heine formulas, Sturm-Liouville problems, algebraic
Darboux transformations.}
\subjclass[2000]{Primary 33C45; Secondary 34B24, 42C05. }

\maketitle

\section{Introduction}

Let $\mu$ be a probability measure supported on an infinite subset $E$ of the real line. We assume that $\int_{E} |x^{n}| d\mu < \infty$ for every nonnegative number $n$.
A sequence of monic polynomials $\{P_{n} (x)\}_{n\geq0}$ is said to be orthogonal with respect to $\mu$ when deg $P_{n} = n$ and $\int_{E} P_{n} (x)P_{m} (x) d\mu(x)= 0$ for $n\neq m$.
It is very well known that these polynomials satisfy a three term recurrence
relation that yields for the orthonormalized polynomials a symmetric tridiagonal
(Jacobi) matrix such that the eigenvalues of the $n$ leading principal submatrix
are the zeros of the polynomial $P_{n}$. As a straightforward consequence of
this fact the zeros of $P_{n}$ are real, simple and interlace with the zeros of
$P_{n-1}$. On the other hand, they are located in the interior of the convex
hull of $E$.\\

The theory of orthogonal polynomials is strongly related to Sturm-Liouville problems. In particular, the so called classical orthogonal polynomials (Hermite, Laguerre, and Jacobi) appear as eigenfunctions of second order linear differential operators with polynomial coefficients. Indeed, the corresponding measure of orthogonality is absolutely continuous and their derivative with respect to the Lebesgue measure (weight function) is the density function of the normal, gamma and beta distributions, respectively. Notice that this fact was pointed out by E. Routh in 1884 \cite{Routh} as well as by S. Bochner in 1929 \cite{Bo}, but the orthogonality does not play therein any role. On the other hand, they are hypergeometric functions and, as a consequence, many analytic properties can be deduced from this fact. Moreover, certain properties of their zeros can be easily deduced using the classical Sturm theorems. Finally, a nice electrostatic interpretation of their zeros is deduced from the second order linear differential equation as an equilibrium problem for the logarithmic interaction of positive unit charges under an external field.\\

Exceptional orthogonal polynomials constitute a recent new approach
to spectral problems for second order linear differential operators
with polynomial eigenfunctions. Previously, a constructive theory of
orthogonal polynomials related to the classical ones has been done
in two directions.  The first one is related to the spectral theory
of higher order linear differential operators with polynomial coefficients. For
fourth order differential operators the classification of their eigenfunctions,
which are sequences of orthogonal polynomials with respect to a nontrivial
probability measure supported on an infinite subset of the real line,
was done by H. L. Krall and A. M. Krall \cite{HKrall,AKrall1,AKrall2} and
essentially
yields the classical ones and perturbations of
some particular Laguerre weights ${\rm e}^{-x} + M \delta(x)$, Jacobi weights
$(1-x)^{\alpha} + M \delta(x)$ and Legendre weight $ 1 + M \delta(x-1) + M
\delta(x+1)$, $M\geq0$. For higher order, some examples are known but a general
theory and
classification constitutes an open problem. The second one appears when some
perturbations of the measure are considered. In particular, three cases are
considered in the literature in the framework of the so called spectral linear
transformations \cite{zhedanov97}. The Christoffel transformation (the multiplication of the
measure by a positive polynomial in the support of the measure), the Uvarov
transformation (the addition of  mass points off the support of the measure) and
Geronimus transformation (the multiplication by the inverse of a positive
polynomial). They can be analyzed in terms of the discrete Darboux
transformation of the corresponding Jacobi matrices using the LU and UL
factorizations and commuting them \cite{BueMar}.\\

Exceptional orthogonal polynomials depart from the classical families in that the sequence of exceptional polynomials is not required to contain a polynomial of every degree, and as a consequence new differential operators exist, with \textit{rational} rather than polynomial coefficients. Despite this fact, the sequence of exceptional polynomial eigenfunctions is still dense in the corresponding weighted $L^2$ space and constitutes an orthogonal polynomial system. The measure of orthogonality for the exceptional families is a classical
measure divided by the square of a polynomial with zeros outside the support of the measure. 

The first explicit examples of families of exceptional orthogonal
polynomials are the $\rX_1$-Jacobi and $\rX_1$-Laguerre polynomials,
which are of codimension one, and were first introduced in
\cite{GKM09,GKM10a}. In these papers, a characterization
theorem was proved for these orthogonal polynomial families, realizing
them as the unique complete codimension one families defined by a 
Sturm-Liouville problem. One of the key steps in the proof was the
determination of normal forms for the flags of univariate polynomials
of codimension one in the space of all such polynomials, and the
determination of the second-order linear differential operators which
preserve these flags \cite{GKM05,GKM12b}.

Shortly after, Quesne \cite{Quesne1,Quesne2}  observed the presence of a
relationship between exceptional orthogonal polynomials and the Darboux
transformation\footnote{By Darboux transformation, we do not mean here the
factorization of Jacobi matrices into upper triangular and lower triangular
matrices mentioned above, but the factorization of the second order
linear differential operator into two first order linear differential operators
\cite{GKM04a,GKM04b}.}. This enabled her to
obtain examples of potentials corresponding to orthogonal polynomial
families of codimension two, as well as explicit families of $\rX_2$
polynomials. Higher-codimensional families were first obtained by
Odake and Sasaki \cite{Sasaki-Odake1}. The same authors further showed the
existence of two families of $\rX_m$-Laguerre and $\rX_m$-Jacobi
polynomials \cite{Sasaki-Odake4}, the existence of which was explained in \cite{GKM10b} for
$\rX_m$-Laguerre polynomials and in \cite{GKM12b} for  $\rX_m$-Jacobi polynomials,
through the application of the isospectral algebraic Darboux
transformation first introduced in \cite{GKM04a,GKM04b}. These exceptional
orthogonal polynomials have been applied in a number of interesting physical
contexts, such as Dirac operators minimally coupled to external fields,
\cite{Ho11}, entropy measures in quantum information theory, \cite{DR12},
rational extensions of Morse and Kepler-Coulomb problems,
\cite{Grandati,Grandati2} or discrete quantum mechanics, \cite{SO6}.

The aim of our contribution is to explore analytic properties of these exceptional polynomials. In particular we will focus our attention in the distribution of their zeros in terms of the support of the orthogonality measure as well as their limit behavior. On the other hand, we will analyze some asymptotic properties as the outer relative asymptotics in terms of the corresponding classical orthogonal polynomials and the Mehler-Heine type formulas. Some properties of the zeros have also been analyzed numerically in \cite{HS}.

\section{Exceptional orthogonal polynomials}
  Let $W(z)$ be a positive weight function with finite moments.  Usually,
  orthogonal polynomials are defined by applying Gram-Schmidt orthogonalization
to the
  standard flag $1,z,z^2,\ldots$ relative to an $L^2$ inner product associated
with the weight $W$. Moreover, if the resulting orthogonal
  polynomials are eigenfunctions of a Sturm-Liouville problem, we
  speak of classical orthogonal polynomials.  By Bochner's theorem,
  the range of such polynomials is limited to the classical families
  of Hermite, Laguerre, and Jacobi (for positive weights) and Bessel (for
signed weights).

  In order to go beyond the classical families, we consider orthogonal
  polynomials spanning a non-standard polynomial flag, say with a
  basis $p_m(z), p_{m+1}(z),\ldots, $ where $\deg p_j = j$.  Once we
  drop the assumption that the OP sequence contains a polynomial \textit{of
  every degree}, we obtain new classes of orthogonal polynomials
  defined by Sturm-Liouville problems, which are commonly referred to
  as exceptional orthogonal polynomials (XOPs)\footnote{Note that the
requirement that the degree sequence starts at $m$ and contains every integer
$j>m$ is not essential either, although all the families treated in this paper
belong
to this class. There exist also XOPs where the degree sequence has gaps. They
are related to state-adding Darboux transformations (as opposed to
isospectral) and contain for instance the $X$-Hermite families, beside many
others.}.

  In the last two years it has become clear that the Darboux
  transformation, appropriately generalized to the polynomial context,
  plays an essential part in the deliniation of XOPs.  To wit, let
  $\cP_n$ denote the vector space of polynomials of degree $\leq n$,
  and consider a codimension $m$ polynomial flag
  \[ \cU = \{ U_k\}_{k=1}^\infty,\quad U_{k-1}\subset U_k \subset
  \cP_{m+k-1}\,.\] Furthermore, let
  \begin{equation}
 A[y] = b(z)(y'-w(z) y),\quad B[y]= \hat{b}(z)(y'-\hat{w}(z)y)
  \end{equation}
  be first order linear differential operators with rational coefficients such
that
 \begin{equation}\label{eq:ABdef1}
   B[\cP_{k-1}] = U_k,\quad A[U_k] = \cP_{k-1},\qquad k=1,2,\dots
 \end{equation}
i.e. $B$ maps the standard flag into the codimension $m$ flag while $A$ maps
the codimension $m$ into the standard flag. Note that equation
\eqref{eq:ABdef1} implies that $\ker A=\ker B=0$ and the Darboux transformation
is isospectral.
Next, consider the second order differential operators
  \begin{equation}
   T = AB,\quad \hat{T} = BA.
  \end{equation}
 By construction, $T$ leaves invariant the standard polynomial flag, while
$\hat{T}$ leaves invariant the codimension-$m$ flag $\cU$.  Hence, by
Bochner's theorem,
  \begin{equation}
   T[y] = p(z) y'' + q(z) y' + r(z) y,\quad \hat{T}[y] = p(z) y''
+
  \hat{q}(z) y' + \hat{r}(z) y
  \end{equation}
 where $p,q,r$ are polynomials with $\deg p \leq 2, \deg q\leq 1, \deg r= 0$
but where $\hat{q},
  \hat{r}$ are in general rational functions.
It is then assured that $T$ and $\hat T$ have polynomial eigenfunctions
\footnote{Since no domains have been specified for $A$ and $B$, the term
\textit{eigenfunction} is not meant in the strict spectral
theoretic sense here, but rather as polynomial solutions to the eigenvalue
equation.} . Let $y_j,\; j\geq 0$, denote the polynomial eigenfunctions of
$T[y]$ and $\hat{y}_j,\; j\geq m$, denote the polynomial eigenfunctions of
$\hat{T}[y]$.
  Again, by construction we have the following intertwining relations
\begin{equation}
  TA = A \hat{T},\quad BT = \hat{T} B,
\end{equation}
 which mean that
  \begin{equation}
    B[y_j] = \beta_j \hat{y}_{j+m},\quad A[\hat{y}_{j+m}] =
\alpha_j
  y_j,\quad j=0,1,2,\ldots,
  \end{equation}
 where $\alpha_j, \beta_j$ are constants, i.e.  operator $A$ maps
eigenfunctions of $T$ into eigenfunctions of $\hat T$ while $B$ does the
opposite transformation.

  Furthermore, let
  \begin{equation}
  W(z) = \frac{1}{p}\exp \int^z \!\!\frac{q}{p},\qquad \hat{W} =
\frac{1}{p}\exp \int^z\!\!
 \frac{ \hat{q}}{p}
  \end{equation}
  be the solutions of the Pearson's equations
  \begin{equation}
    \label{eq:pearson}
    (p W)' = q W,\quad (p\hat{W})' = \hat{q} \hat{W}.
  \end{equation}
  This means that $T[y]$ is formally self-adjoint relative to $W$ while
$\hat{T}[y]$ is formally self-adjoint relative to $\hat{W}$.  Consequently, the
eigenpolynomials $y_j$ are formally
  orthogonal with respect to the weight $W(z)$, while $\hat{y}_j$ are
  formally orthogonal relative to $\hat{W}(z)$.  One can also
show
  that
  \[ \int A[f] g W dz + \int B[g] f \hat{W} dz = \text{ boundary
    term};\] which means that  operators $A$ and $-B$ are formally adjoint.
  By a careful choice of the flags, and by imposing appropriate
  boundary conditions one can construct examples where the above
  formal relations hold in the $L^2$ setting (i.e. boundary conditions such that
the boundary terms vanish) and thereby obtain novel
  classes of exceptional orthogonal polynomials.

In the present note we study asymptotic behaviour of XOPs of Laguerre and
Jacobi types. As we  show, in the interval of orthogonality the exceptional
polynomials satisfy a variant of the classical Heine-Mehler formula.  Outside
the interval of orthogonality, the convergence picture is less clear.  However,
one can show that codimension $m$ exceptional orthogonal polynomials  possess
$m$ extra zeros outside the interval of orthogonality, which we shall denote as
\textit{exceptional zeros}. These exceptional zeros have well-defined
convergence behaviour, and they converge to the zeros of some fixed
classical orthogonal polynomial.

\section{Type I Exceptional Laguerre polynomials}
Let us illustrate the above discussion with the particular example of
the so-called type I exceptional Laguerre polynomials.  Let
$\Lag{\al}{n}(z)$ denote the classical Laguerre polynomial of
degree $n$ and
\[ \cL_\al[y] =zy'' + (\al+1-z) y' \] the classical Laguerre operator.
Thus, $y=\Lag{\al}{n}$ is the unique polynomial solution of the equation
\[ \cL_\al[y] = - n y, \quad\text{ with the normalization }\quad
y^{(n)}(0)=(-1)^n.\]
An equivalent boundary condition  is
\begin{equation}
  \label{eq:Lz=0}
  \Lag{\al}{n}(0) = \binom{n+\al}{n}.
\end{equation}

For a fixed non-negative integer $m\geq0$, let us now define
\begin{align}
 \label{eq:xidef} \xi_{\al,m}(z) &= \Lag{\al}{m}(-z),\\
  A^{\I}_{\al,m}[y]  &= \xi_{\al,m} y'-\xi_{\al+1,m} y,\\
  B^{\I}_{\al,m}[y] &= \frac{z y' + (1+\al)\,y}{\xi_{\al,m}},\\
  \XLagI{\al}{m}{m+j} &= -A^{\I}_{\al-1,m}\left[\Lag{\al-1}{j}\right],\qquad
j=0,1,2,\dots\\
  \cL^{\I}_{\al,m}[y] &= \cL_\al[y]+my -2
  (\log\xi_{\al-1,m})'\left(zy'+\al y\right).
\end{align}
The following factorization relations follow from standard Laguerre
identities:
\begin{align}
  \cL_{\al} &= B^{\I}_{\al,m} A^{\I}_{\al,m} +\al+m+1,\\
  \cL^{\I}_{\al,m} &= A^{\I}_{\al-1,m} B^{\I}_{\al-1,m} + \al+m.
\end{align}
Consider now the following codimension $m$ polynomial subspace
\begin{equation}
 U^\I_{\al,j} = \{ f \in \cP_{j+m-1} : \xi_{\al-1,m} | (zf' +\al
f)\},\qquad j=1,2,\ldots,
\end{equation} where $f|g$ means polynomial $f(z)$
divides polynomial $g(z)$.  At the level of flags, the above
factorizations correspond to the following linear isomorphisms:
\begin{equation}
B^\I_{\al-1,m}:U^\I_{\al,j}\to\cP_{j-1},\qquad
A^\I_{\al-1,m}:\cP_{j-1}\to U^\I_{\al,j}.
\end{equation}

The polynomials $\left\{\XLagI{\al}{m}{m+j}\right\}_{j=0}^\infty$ are known in
the literature as the\textit{ type I exceptional codimension $m$ Laguerre
polynomials} (for short,
type I $X_m$-Laguerre) \cite{Sasaki-Odake1,GKM10b}.  By construction, the
$X_m$-Laguerre polynomials have
the following properties:
\begin{itemize}
\item they span the flag $ U^\I_{\al,1}\subset U^\I_{\al,2} \subset\cdots$
\item they satisfy the following second order linear differential equation:
  \begin{equation}
   \cL^\I_{\al,m}\left[ \XLagI{\al}{m}{m+j} \right]=- j
\XLagI{\al}{m}{m+j},\qquad j\geq 0,
  \end{equation}
\item they are orthogonal with respect to the weight
  \begin{equation}
    \label{eq:WIkmdef}
    W^{\I}_{\al,m}(z) := \frac{z^\al
e^{-z}}{\xi_{\alpha-1,m}(z)^2}=\frac{z^\al
e^{-z}}{\left[\Lag{\al-1}{m}(-z)\right]^2},\quad z\in [0,\infty),
  \end{equation}

\item they are dense in the Hilbert space $\rL^2
\big([0,\infty),W^{\I}_{\al,m}\big)$.
\end{itemize}

Note that for $\al\geq 0$ the polynomial
$\xi_{\al-1,m}(z)=\Lag{\al-1}{m}(-z)$ in the denominator of the weight
has its zeros on the negative real axis, and hence $W^\I_{\al,m} dz$ is a
positive definite measure in $\mathbb R^+$, with well defined
moments of all orders.  As a result, for $\al\geq 0$ the set $\{
\XLagI{\al}{m}{n}\}_{n=m}^\infty$ constitutes an orthogonal polynomial
basis of $\rL^2 \big([0,\infty),W^{\I}_{\al,m}\big)$. We observe that
the last property of the above list does not follow by the algebraic
construction and needs to be established by a separate argument. The
interested reader is referred to
\cite{GKM10b} for a direct proof of the
completeness of the $X_m$-Laguerre families.

\begin{prop}
  The type I $X_m$-Laguerre polynomials $\XLagI{\al}{m}{n}$ can be expressed
in terms of classical associated Laguerre polynomials \emph{with the same
parameter}
$\al$ as follows
  \begin{equation}
    \label{eq:L1altrep}
    \XLagI{\al}{m}{m+j}= \xi_{\al,m} \Lag{\al}{j} - \xi_{\al,m-1}
\Lag{\al}{j-1} ,\qquad j\geq 0.
  \end{equation}
\end{prop}
The above representation will be specially useful to
discuss the asymptotic properties of the zeros of type I $X_m$-Laguerre
polynomials. It is clear from it that $\XLagI{\al}{m}{m+j}$ has degree $m+j$.
This representation is reminiscent of the expansions obtained by rational
modifications of classical weights in the framework of
spectral linear transformations (see \cite{zhedanov97}). However, it should be
stressed that they are essentially different because in the case of exceptional
polynomials, although there is a rational modification of the weight, we
are not dealing with the standard flag.

\begin{proof}
  Using elementary identities, we have
  \begin{align*}
    0&=\xi_{\al,m}(\Lag{\al}{j} - \Lag{\al}{j-1} - \Lag{\al-1}{j}) +
\Lag{\al}{j-1}
    (\xi_{\al,m}-\xi_{\al,m-1} - \xi_{\al-1,m}) \\
    &= -(\xi_{\al,m} \Lag{\al-1}{j} + \xi_{\al-1,m}
    \Lag{\al}{j-1})+(\xi_{\al,m}\Lag{\al}{j} -\xi_{\al,m-1} \Lag{\al}{j-1}).
  \end{align*}
  Therefore, we can re-express the type I $X_m$-Laguerre polynomials as
  \begin{align*}
    \XLagI{\al}{m}{m+j} &= \xi_{\al,m} \Lag{\al-1}{j} - \xi_{\al-1,m}
{\Lag{\al-1}{j}}'\\
    &=  \xi_{\al,m} \Lag{\al-1}{j} +\xi_{\al-1,m} \Lag{\al}{j-1}\\
    &= \xi_{\al,m} \Lag{\al}{j} - \xi_{\al,m-1} \Lag{\al}{j-1}.
  \end{align*}
\end{proof}

We are now ready to prove an interlacing result for the zeros of type I
$X_m$-Laguerre polynomials, but before let us recall the following classical
identity
\begin{equation}
  \label{eq:Lat0}
  \Lag{\al}{n}(0) = (\al+1)_n /n!
\end{equation}
where 
\[(x)_n  = \begin{cases}  x(x+1) \cdots (x+n-1)  & \text{if $n\geq0$,}
\\
1 &\text{if $n= 0$,}
\end{cases}\] 
is the usual Pochhammer symbol.
Using the above representation we obtain an analogous expression for
the type I $X_m$-Laguerre polynomials:
\begin{equation}
  \label{eq:L1at0}
  \XLagI{\al}{m}{m+j}(0) =  \frac{\al+j+m}{k} \frac{(\al)_m}{m!
}\frac{(\al)_j}{j!}
\end{equation}

\begin{prop}
  For $\al>0$ the type I exceptional Laguerre polynomial
$\XLagI{\al}{m}{m+j}(z)$ has $j$ simple zeros
  in $z\in (0,\infty)$ and $m$ simple zeros in $z\in (-\infty,0)$. The
  positive zeros of $\XLagI{\al}{m}{m+j}(z)$ are located between
  consecutive zeros of $\Lag{\al}{j}$ and $\Lag{\al}{j-1}$ with the smallest
  positive zero of $\XLagI{\al}{m}{m+j}(z)$ located to the left of the
  smallest zero of $\Lag{\al}{j}$.  The negative zeros of
  $\XLagI{\al}{m}{m+j}(z)$ are located between the consecutive zeros of
  $\xi_{\al,m-1}$ and $\xi_{\al,m}$.
\end{prop}
\begin{proof}
  Let $0<\ezeta{\al}{n}{1}<\ezeta{\al}{n}{2}<\cdots < \ezeta{\al}{n}{n}$ denote
  the zeros of $\Lag{\al}{n}(z)$ listed in increasing order.  According to the
interlacing property of the zeros of classical orthogonal polynomials, we have
  \[ \ezeta{\al}{j}{1} < \ezeta{\al}{j-1}{1} < \ezeta{\al}{j}{2} <
\ezeta{\al}{j-1}{2}
  < \cdots < \ezeta{\al}{j-1}{j-1} < \ezeta{\al}{j}{j}.\] Recall that
  $\Lag{\al}{n}(z)>0$ for $z\leq 0$.  This implies that $\xi_{\al,m}(z),
  \xi_{\al,m-1}(z)>0$ for $z\geq0$.  Hence, by \eqref{eq:L1altrep},
  \begin{gather*}
    \sgn \XLagI{\al}{m}{m+j}(\ezeta{\al}{j}{i}) =  (-1)^{i},\quad i=1,\ldots,
j;\\
    \sgn \XLagI{\al}{m}{m+j}(\ezeta{\al}{j-1}{i}) = (-1)^{i},\quad i=1,\ldots,
    j-1.
  \end{gather*}
  It follows by \eqref{eq:L1at0} that there is a zero of
  $\XLagI{\al}{m}{m+j}$ in the interval $(0,\ezeta{\al}{j}{1})$ and a zero in
  the interval $(\ezeta{\al}{j-1}{i-1},\ezeta{\al}{j}{i})$ for every
  $i=2,\ldots, j$. An analogous argument places zeros of
  $\XLagI{\al}{m}{m+j}$ in the intervals $(-\ezeta{\al}{m}{1},0)$ and
  $(-\ezeta{\al}{m}{i},-\ezeta{\al}{m-1}{i-1}),\; i=2,\ldots, m$.  By
  exhaustion, each of the above intervals contains one simple zero of
  $\XLagI{\al}{m}{m+j}$.
\end{proof}

We now study the distribution of the zeros of $\XLagI{\al}{m}{n}$ as
$n\to \infty$.  To that end, we will use the classical Heine-Mehler
formula for Laguerre polynomials
\begin{equation}
  \label{eq:hmclassic}
  \Lag{\al}{n}\left(z/n\right) n^{-\al} \rightrightarrows z^{-\al/2}
  J_{\al}(2\sqrt{z}),\qquad n\to \infty,
\end{equation}
where $J_\al(z)$ denotes the Bessel function of the first kind of order $\al$
$(\al>-1)$ and the double arrow denotes uniform
convergence in compact domains of the complex plane.

The exceptional Laguerre polynomials admit a generalization of the classical
Heine-Mehler formula, given by the following:

\begin{prop}[Generalized Heine-Mehler formula]
  We have
  \begin{equation}
    \label{eq:hmL1}
    \XLagI{\al}{m}{n}\left(z/n\right) n^{-\al} \rightrightarrows
\binom{\al+m-1}{m}
    z^{-\al/2} J_\al(2\sqrt{z}),\qquad n\to\infty.
  \end{equation}

\end{prop}

A numerical representation of the convergence of the scaled exceptional
Laguerre polynomials to the Bessel function is given in Figure \ref{fig:HM}.

\begin{proof}
 Multiply  \eqref{eq:L1altrep} by $j^{-\al}$ and replace
$z\to z/j$. Taking the limit $j\to\infty$ and using the classical
Heine-Mehler formula \eqref{eq:hmclassic} leads to
\[
j^{-\al} \XLagI{\al}{m}{m+j} \rightrightarrows
\big(\xi_{\al,m}(0)-\xi_{\al,m-1}(0)\big)z^{-\al/2}J_\al(2\sqrt{z}).
\]
The final expression \eqref{eq:hmL1} is recovered by noting that
\[\xi_{\al,m}(0)-\xi_{\al,m-1}(0)=\binom{\al+m-1}{m}\]
as implied by \eqref{eq:xidef} and \eqref{eq:Lat0}.
\end{proof}

Note that for $m=0$ the classical Heine-Mehler formula is recovered as a
particular case.

\begin{figure}[h]\label{fig:HM}
\includegraphics[width=0.75\textwidth]{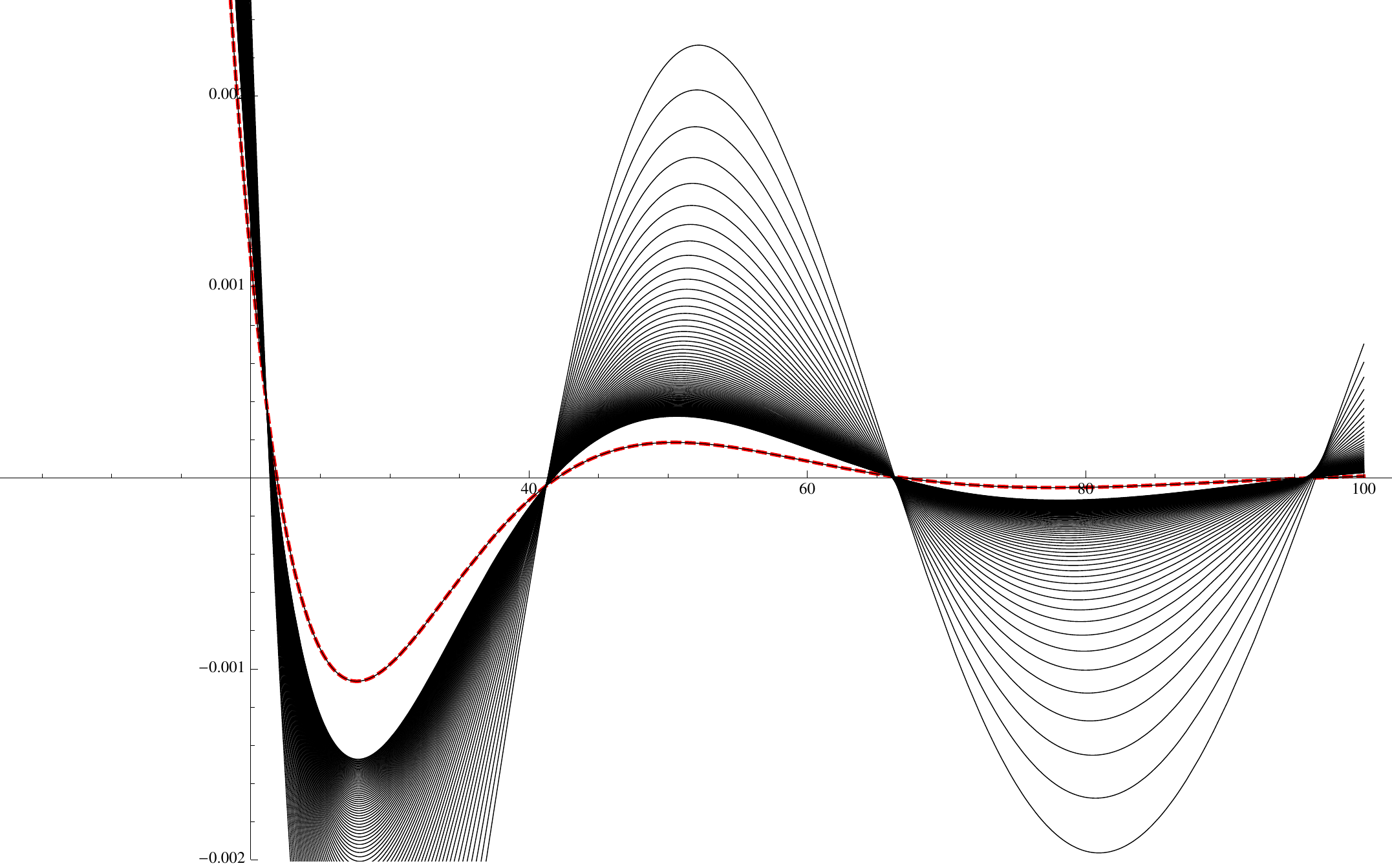}
\caption{Plot of $n^{-\al}\XLagI{n}{m}{\al}(x/n)$ for $m=3$, $\al=5.5$ and
$20\leq n\leq 100$. The dashed red line corresponds to the limiting
Bessel function $x^{-\al/2} \binom{\al+m-1}{m} J_\al(2\sqrt{x})$ predicted by
the generalized Heine-Mehler formula \eqref{eq:hmL1}.}
\end{figure}

\begin{cor} Let $\left\{\tilde z^{(\al)}_i\right\}_{i\geq 1}$ be the sequence
of zeros of the Bessel function $J_\al(z)$ listed in increasing order and let
$\{x_{j,i}\}_{i=1}^j$ denote the regular
zeros of $\XLagI{\al}{m}{m+j}(z)$ in the interval $[0,\infty)$. Then we get
the following asymptotic behaviour
 \begin{equation}\label{eq:asymbessel} \lim_{j\to\infty} j x_{j,i} =
\frac{(\tilde z^{(\al)}_i)^2}{4}.
\end{equation}
\end{cor}
\begin{proof}
The above result follows from \eqref{eq:hmL1} and Hurwitz's
theorem.
\end{proof}
We have already seen that for fixed $m$ the asymptotic behaviour of the regular
zeros of the exceptional Laguerre polynomials coincides with that of the
classical Laguerre. We now investigate in the same limit the behaviour of the
$m$ exceptional zeros of
$\XLagI{\al}{m}{m+j}(z)$.

\begin{prop}
  As $j\to \infty$ the $m$ zeros of
  $\XLagI{k}{m}{m+j}$ in $(-\infty,0)$ converge to the $m$ zeros of
$\Lag{\al-1}{m}(-z)$.
\end{prop}
\begin{proof}
  The classical Laguerre polynomials have the following outer ratio asymptotics
  \begin{equation}
    \label{eq:Lora}
    \frac{\Lag{\al}{n+1}(z)}{ \Lag{\al}{n}(z)} \rightrightarrows 1,\qquad
z\notin
[0,\infty),\qquad n\to\infty.
  \end{equation}
Hence, by \eqref{eq:L1altrep}, we have for $z\notin [0,\infty)$,
  \begin{equation*}
    \frac{\XLagI{\al}{m}{m+j}}{\Lag{\al}{j}} =\xi_{\al,m}
    -\xi_{\al,m-1}\frac{\Lag{\al}{j-1}}{\Lag{\al}{j}}\rightrightarrows
\xi_{\al,m}-\xi_{\al,m-1}
    = \xi_{\al-1,m}.
  \end{equation*}
  Therefore, by Hurwitz's theorem the $m$ exceptional zeros of
  $\XLagI{\al}{m}{m+j}$ converge to the zeros of
$\xi_{\al-1,m}=\Lag{\al-1}{m}(-z)$.
\end{proof}

\begin{figure}[h]\label{fig:Lag1zeros}
\begin{tabular}{cc}
\includegraphics[width=0.45\textwidth]{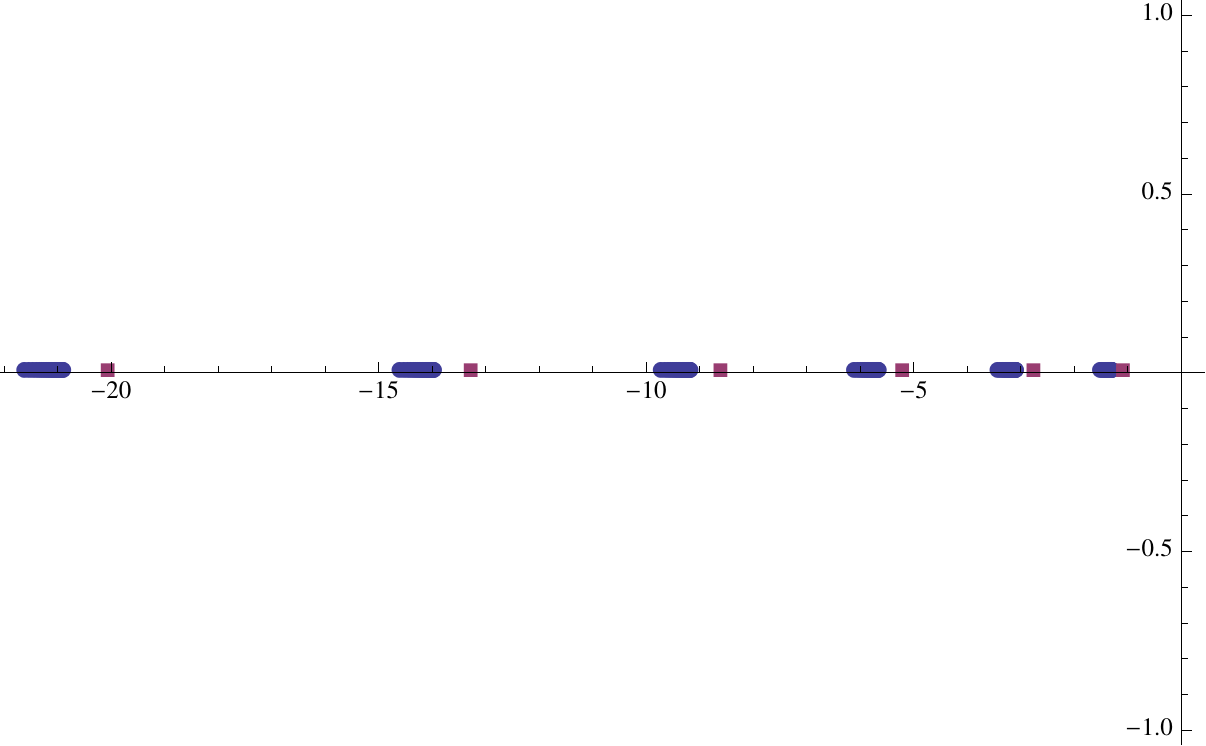} &
\includegraphics[width=0.45\textwidth]{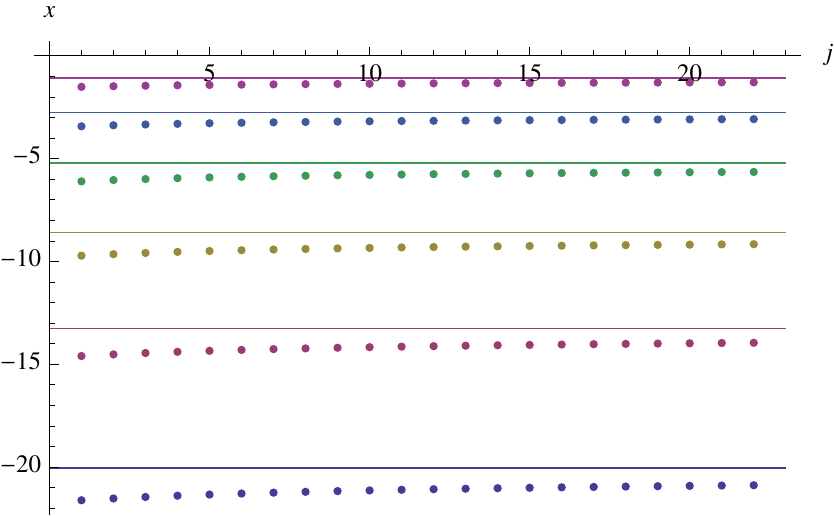}
\end{tabular}
\caption{\textit{Left}: Exceptional zeros of the polynomials
$\XLagI{\al}{m}{m+j}(z)$ for
$m=6$, $\al=3.5$ and $1\leq j\leq 22$. The squares denote the zeros of
$\Lag{\al-1}{m}(-z)$ to which the zeros of $\XLagI{\al}{m}{m+j}(z)$
converge for $j\to\infty$. \textit{Right}: Convergence is better seen by
plotting the position of the negative zeros as a function of $j$. Solid lines
show the limiting values.}
\end{figure}


%

\section{Type II Exceptional Laguerre polynomials}
\subsection{Definition and identities.}
Let $m\geq 0$ be an integer and $\al$ a real number. Let us introduce
the polynomials
\begin{align}
  \eta_{\al,m}(z) := \Lag{-\al}{m}(z).
\end{align}

For a fixed non-negative integer $m$ and a real number $\alpha$ let us define
the following first and second order operators
\begin{align}
  \label{eq:A2def}
  A^{\II}_{\al,m}[y]&:= z \eta_{\al,m} y'+(\al-m)\eta_{\al+1,m} y,\\
  B^{\II}_{\al,m}[y]&:= (y'-y)/\eta_{\al,m},\\
  \cL[y] &:= z y'' + (\al+1-z) y',\\
  \label{eq:L2def}
  \cL^{\II}_{\al,m}[y] &:= \cL[y] + 2z
  (\log\eta_{\al+1,m})' (y-y') -my.
\end{align}
The following factorizations follow from standard Laguerre identities:
\begin{align}
  \label{eq:B2A2}
  \cL_\al &= B^{\II}_{\al,m} A^{\II}_{\al,m} +\al-m,\\
  \label{eq:A2B2}
  \cL^{\II}_{\al,m}&= A^{\II}_{\al+1,m} B^{\II}_{\al+1,m}+\al+1-m.
\end{align}
For a real number $\alpha$ and given integers $n\geq m\geq 0$, we
define the $n^{\text{th}}$ degree type II exceptional Laguerre polynomial by
\begin{align}
  \label{eq:L2ndef}
  \XLagII{\al}{m}{n}(z) &:= -A^{\II}_{\alpha+1,m}[\Lag{\al+1}{j}],\qquad
  j=n-m\geq 0.
\end{align}
Expanding \eqref{eq:A2def} and applying standard identities,
the following dual representations of the type II polynomials hold:
\begin{align}
  \label{eq:L2ndef1}
  \XLagII{\al}{m}{m+j} &= z\Lag{-\al-1}{m} \Lag{\al+2}{j-1}
  +(m-\al-1)\Lag{-\al-2}{m} \Lag{\al+1}{j} \\
  \label{eq:L2ndef2}
  &=-z\Lag{-\al}{m-1} \Lag{\al+1}{j}
  -(\al+1+j)\Lag{-\al-1}{m} \Lag{\al}{j}.
\end{align}
Hence, up to a constant
factor, the type II polynomials extend their classical counterparts, which are
recovered for the particular case $m=0$:
\begin{equation}
  \label{eq:L2m=0}
  \XLagII{\al}{0}{n} = -(1+\al+n) \Lag{\al}{n}.
\end{equation}
The factorizations \eqref{eq:B2A2} and \eqref{eq:A2B2} yield the following
intertwining  relations between the standard Laguerre operator $\cL_\al$ and
the type II $X$-Laguerre operator $\cL^{\II}_{\al,m}$:
\begin{align}
  \label{eq:LIIA}
  &\cL^{\II}_{\al,m} A^{\II}_{\al+1,m}  = A^{\II}_{\al+1,m}\cL_{\al+1},\\
  \label{eq:BLII}
  &B^{\II}_{\al+1,m} \cL^{\II}_{\al,m} = \cL_{\al+1} B^{\II}_{\al+1,m}.
\end{align}
The former relation provides the eigenvalue relation for the type II
exceptional Laguerre polynomials:
\begin{equation}
  \label{eq:IIeigenval}
  \cL^{\II}_{\al,m} \left[\XLagII{\al}{m}{m+j}\right]= -j
\XLagII{\al}{m}{m+j},\qquad   j=0,1,2\dots.
\end{equation}
The latter gives the following ``lowering'' relation between the type
II exceptional Laguerre polynomials and their classical counterparts:
\begin{equation}
  \label{eq:L2L}
  \XLagII{\al}{m}{n}\,{}' - \XLagII{\al}{m}{n} = (1+n-\al) \Lag{-\al-1}{m}
  \Lag{\al+1}{j},\quad j=n-m\geq 0
\end{equation}
In order to find raising and lowering relations for the exceptional
Laguerre polynomials, let us introduce the following first order linear
differential operators
\begin{align}\label{sifact1}
  \hA^{\II}_{\al,m}[y] &:= \frac{\eta_{\al+2,m}}{\eta_{\al+1,m}} \left(
    y' -(\log \eta_{\al+2,m})'  y\right),\\
  \hB^{\II}_{\al,m}[y]&:= \frac{\eta_{\al+1,m}}{\eta_{\al+2,m}}
  \left(zy'+(\al+1-z) y\right) - z(\log \eta_{\al+2,m})' y.
\end{align}
In terms of these operators we have the following \textit{shape-invariant}
factorizations
\begin{align}
  \cL^{\II}_{\al,m} &= \hB^{\II}_{\al,m} \hA^{\II}_{\al,m},\\
  \cL^{\II}_{\al+1,m} &= \hA^{\II}_{\al,m} \hB^{\II}_{\al,m}+1.\label{sifact2}
\end{align}
From these factorizations the following lowering and raising
relations for the exceptional polynomials easily follow:
\begin{align}
  &\hA^{\II}_{\al,m}\left[\XLagII{\al}{m}{n}\right]  =
  \XLagII{\al+1}{m}{n-1},\quad n\geq m+1,\\
  &\hB^{\II}_{\al,m}\left[\XLagII{\al+1}{m}{n}\right] = (n-m)
  \XLagII{\al}{m}{n+1},\quad n\geq m.
\end{align}
The above equations can be conveniently re-written as
\begin{align}
  \label{eq:L2lower}
  & \left( \frac{\XLagII{\al}{m}{n}}{\eta_{\al+2,m}}\right)' =
  \left(\frac{\eta_{\al+1,m}}{\eta_{\al+2,m}}\right)^2
  \frac{\XLagII{\al+1}{m}{n-1}}{\eta_{\al+1,m}}\,,\\
  \label{eq:L2raise}
  &\left( \frac{e^{-z}
      z^{\al+1}}{\eta_{\al+1,m}}\XLagII{\al+1}{m}{n}\right)' = (n-m+1)
  \left(\frac{\eta_{\al+2,m}}{\eta_{\al+1,m}}\right)^2
  \frac{e^{-z} z^{\al}}{\eta_{\al+2,m}}\XLagII{\al}{m}{n+1}\,.
\end{align}

\subsection{Orthogonality}
The type II exceptional Laguerre polynomials are formally orthogonal
with respect to the weight
\begin{equation}
  W^{\II}_{\al,m}(z) :=\frac{ e^{-z} z^\al}{\eta_{\al+1,m}^2}= \frac{ e^{-z}
z^\al}{\left[\Lag{-\al-1}{m}(z)\right]^2}.
\end{equation}
The above weight is the solution $W^{\II}= \hW$ of Pearson's equation
\eqref{eq:pearson} where
\[ p=z,\quad \hat{q} = (1+\al-z) -2 z(\log \eta_{\al+1,m})' \] are extracted
from \eqref{eq:L2def}.  As a consequence, \eqref{eq:IIeigenval} and
Green's formula imply
\begin{equation}
 \label{eq:IIgreen}
 (n_2-n_1)\int \XLagII{\al}{m}{n_2}\,\XLagII{\al}{m}{n_1}\,
 W^{\II}_{\al,m}\, dz = z W^{\II}_{\al,m}\,
 \Wr\left[ \XLagII{\al}{m}{n_2}, \XLagII{\al}{m}{n_1}\right],
\end{equation}
where $n_2>n_1\geq m$ and where
\[\Wr[f,g] = f'g - fg'\]
denotes the usual Wronskian operator. The following
crucial  result is established in \cite[Ch. 6.73]{Sz}.
\begin{prop}
  \label{prop:etazeros}
  For $\al>m-1$ the polynomials $\eta_{\al,m}(z)$ have no zeros
  in $[0,\infty)$.  The number of negative real zeros is either $0$ or $1$
  according to whether  $m$ is even or odd, respectively.
\end{prop}
\noindent
Thus, assuming $\alpha>m-1$ and restricting the interval
of orthogonality to $[0,\infty)$, $W^{\II}_{\al,m}$ is a weight with
finite moments of all orders, and the RHS of \eqref{eq:IIgreen}
vanishes, whch ensures genuine orthogonality in the $L^2$ sense.

\subsection{Zeros of the type II Laguerre polynomials}
Henceforth, let us assume that $\alpha > m-1$, where $m\geq 0$ is an
integer.  As above, we will call the real positive zeros of
$\XLagII{\al}{m}{n}(z),\; n\geq m$ \emph{regular} and the negative and
complex zeros \emph{exceptional}.  From \eqref{eq:L2ndef1} we have
\begin{equation}
  \label{eq:L2z=0}
  \XLagII{\al}{m}{m+j}(0) = (m+1) \binom{\al+j+1}{j} \binom{m-\al-1}{m+1}
  ,\quad j=0,1,2,\ldots
\end{equation}
Hence, $z=0$ is never a zero of such a polynomial.
\begin{prop}
  The zeros of $\XLagII{\al}{m}{n}(z),\; n\geq m$, are simple.
\end{prop}
\begin{proof}
  This follows by \eqref{eq:L2ndef1} \eqref{eq:L2ndef2} \eqref{eq:L2L}
  and the fact that the zeros of the classical Laguerre polynomials are simple.
\end{proof}
\begin{prop}
  \label{prop:L2zeros}
  The polynomial $\XLagII{\al}{m}{n}(z),\; n\geq m$, has exactly $n-m$
  regular
  zeros.
\end{prop}
\begin{proof}
  We prove the existence of at least $j=n-m$ regular zeros by
  induction on $j$.  The case $j=0$ is trivial.  Suppose now that the
  proposition has been established for $j\geq 0$ and $\al>m-1$.  Since
  $\al+1>m-1$, the proposition is also true for
  $\XLagII{\al+1}{m}{n}$. Let $\zeta_1,\ldots, \zeta_k,\; k\geq j$,
  be the regular zeros of $\XLagII{\al+1}{m}{n}(z)$.  By
  \eqref{eq:L2raise} and Rolle's theorem, $\XLagII{\al}{m}{n+1}(z)$
  has at least one zero in each of the intervals $(\zeta_1,\zeta_2),
  (\zeta_2,\zeta_3),\ldots, (\zeta_{k-1},\zeta_k)$.  Also by
  \eqref{eq:L2raise}, there is a zero in $(0,\zeta_1)$ and a zero in
  $(\zeta_i,\infty)$, for a total of at least $k+1$ zeros.

  We conclude by showing that $j$ is also an upper bound for the
  number of regular zeros.  The proof is again by induction on $j=n-m$.
  By \eqref{eq:L2ndef1},
  \begin{equation}
    \label{eq:L2ground}
    \XLagII{\al}{m}{m} = (m-1-\al)\Lag{-\al-1}{m}  = (m-1-\al)\eta_{\al+1,m}.
  \end{equation}
  by Proposition \ref{prop:etazeros}, the latter has no real,
  non-negative zeros. The lowering relation \eqref{eq:L2lower} shows
  that between two regular zeros of $\XLagII{\al}{m}{n}$  at least one
  zero of $\XLagII{\al+1}{m}{n-1}(z)$ lies.  Hence, if we assume that the latter
  has at most $j-1$ regular zeros, then the former has at most $j$
  regular zeros.
\end{proof}

\begin{prop}
  The type II polynomial $\XLagII{\al}{m}{n},\; n\geq m$, has either 0 or 1
  negative zeros, according to whether $m$ is even or odd.
\end{prop}
\begin{proof}
  Let
  \[\epsilon_{m} =
  \begin{cases}
    0 & m \text{ even}\\
    1 & m \text{ odd.}
  \end{cases}\] Let $\epsilon_{\al,m,n}$ be the number of negative
  zeros of $\XLagII{\al}{m}{n}$.  We wish to show that
  \[ \epsilon_{\al,m,n} = \epsilon_m,\quad \al>m-1,\; n\geq m.\]
  Suppose that $m$ is odd.  By \eqref{eq:L2z=0},
  \[ \sgn \XLagII{\al}{m}{n}(0) = (-1)^{m+1}.\]
  By \eqref{eq:L2ndef1},
  \[ \XLagII{\al}{m}{n} =  \frac{m-1-j-\al}{m! j!} (-z)^{m+j} +
  \text{lower degree terms.} \]
  As a consequence,
  \[ \lim_{z\to -\infty} \sgn \XLagII{\al}{m}{n}(z) = -1 .\] Hence,
  $\epsilon_{\al,m,n} \geq 1$ if $m$ is odd, and therefore
  $\epsilon_{\al,m,n} \geq \epsilon_m$.

  By Proposition \ref{prop:L2zeros}, $\XLagII{\al}{m}{m+j}$ has
  $j+\epsilon_{\alpha,m,n}$ real zeros.  Hence, by \eqref{eq:L2lower},
  $\XLagII{\al+1}{m}{m+j-1}$ has at least $j-1+\epsilon_{\alpha,m,n}$
  real zeros.  Continuing inductively, $\XLagII{\al+j}{m}{m}$ has at
  least $\epsilon_{\al,m,n}$ real zeros.  Hence, by
  \eqref{eq:L2ground} and Proposition \ref{prop:etazeros},
  $\epsilon_{\al,m,n} \leq \epsilon_m$, as was to be shown.
\end{proof}

For the type II exceptional Laguerre polynomials, a Heine-Mehler
type formula also holds:
\begin{prop}
  As $n\to \infty$, we have
  \begin{equation}
    \label{eq:hmL2}
    \XLagII{\al}{m}{n}(z/n)n^{-\al-1} \rightrightarrows
    -\binom{m-1-\al}{m}
    z^{-\al/2} J_\al(2\sqrt{z}).
  \end{equation}
\end{prop}
\begin{proof}
  Multiplying \eqref{eq:L2ndef2} by $j^{-1-\al}$, replacing  $z\to z/j$,
  and applying the classical
  Heine-Mehler formula, we get
  \begin{gather*}
    j^{-1-\al} \XLagII{\al}{m}{m+j}(z/j) +z^{\al/2} J_\al(2\sqrt{z})
    \Lag{-\al-1}{m}(z/j)  \rightrightarrows 0,\quad j\to \infty.
  \end{gather*}
  The polynomials $\XLagII{\al}{m}{m+j}$ and $\Lag{-\al-1}{m}$ are uniformly
  continuous on compact subsets of $\Cset$.  Setting $n=m+j$, we have
  by uniform continuity on compact subsets
  \begin{gather*}
    \Lag{-\al-1}{m}(z/n)\rightrightarrows \binom{m-\al-1}{m} \\
    j^{-1-\al} \XLagII{\al}{m}{m+j}(z/j) -n^{-1-\al}
  \XLagII{\al}{m}{n}(z/n)
    \rightrightarrows 0
  \end{gather*}
  as $n\to \infty$.  Equation \eqref{eq:Lz=0} is needed to establish
  the first statement.
\end{proof}
\noindent
Note that, as a consequence of \eqref{eq:L2m=0}, the above assertion reduces to
the classical Heine-Mehler formula for $m=0$.

\begin{cor} Let $0<\tz_1< \tz_2< \tz_3< \cdots$ denote the
  positive zeros of the Bessel function of the first kind $J_\al(z)$ arranged
in an increasing order
  and let $0<z_{n,1}< z_{n,2} <\cdots < z_{n,m-n}$ be the regular
  zeros of $\XLagII{\al}{m}{n}$ also arranged in increasing order.Then,
  \begin{equation}
    \label{eq:asymbessel2}
    \lim_{n\to\infty} n z_{n,i} = \tz^2_i/4
\end{equation}
\end{cor}
\begin{proof}
The above result follows from \eqref{eq:hmL2} and Hurwitz's
theorem.
\end{proof}

Away from the interval of orthogonality, we can describe the asymptotic
behaviour as follows:
\begin{prop}
  As $j\to \infty$ we have
  \[-(\al+1+j)^{-1} \frac{\XLagII{\al}{m}{m+j}(z)}{\Lag{\al}{j}(z)} =
  \eta_{\al+1,m}(z)
   + O(j^{-1/2})\] on compact subsets of
  $\Cset/[0,\infty)$.
\end{prop}
\begin{proof}
  For the outer ratio asymptotics of the classical Laguerre
  polynomials, we have
  \[ (-z)^{t/2} \frac{\Lag{\al+t}{j}(z)}{
    \Lag{\al}{j}(z)} = O(j^{t/2}),\quad j\to \infty \] uniformly on compact
subsets of $\Cset/[0,\infty)$.  The desired
  conclusion now follows by \eqref{eq:L2ndef2}.
\end{proof}

\begin{prop}
  As $n\to \infty$ the exceptional zeros of $\XLagII{\al}{m}{n},\; n\geq m$,
  converge to the zeros of $\eta_{\al+1,m}(z)=\Lag{-\al-1}{m}(z)$.
\end{prop}
\begin{proof}
  The desired conclusion follows by the preceding Proposition and by
  Hurwitz's theorem.
\end{proof}

\begin{figure}[h]\label{fig:Lag2zeros}
\includegraphics[width=0.75\textwidth]{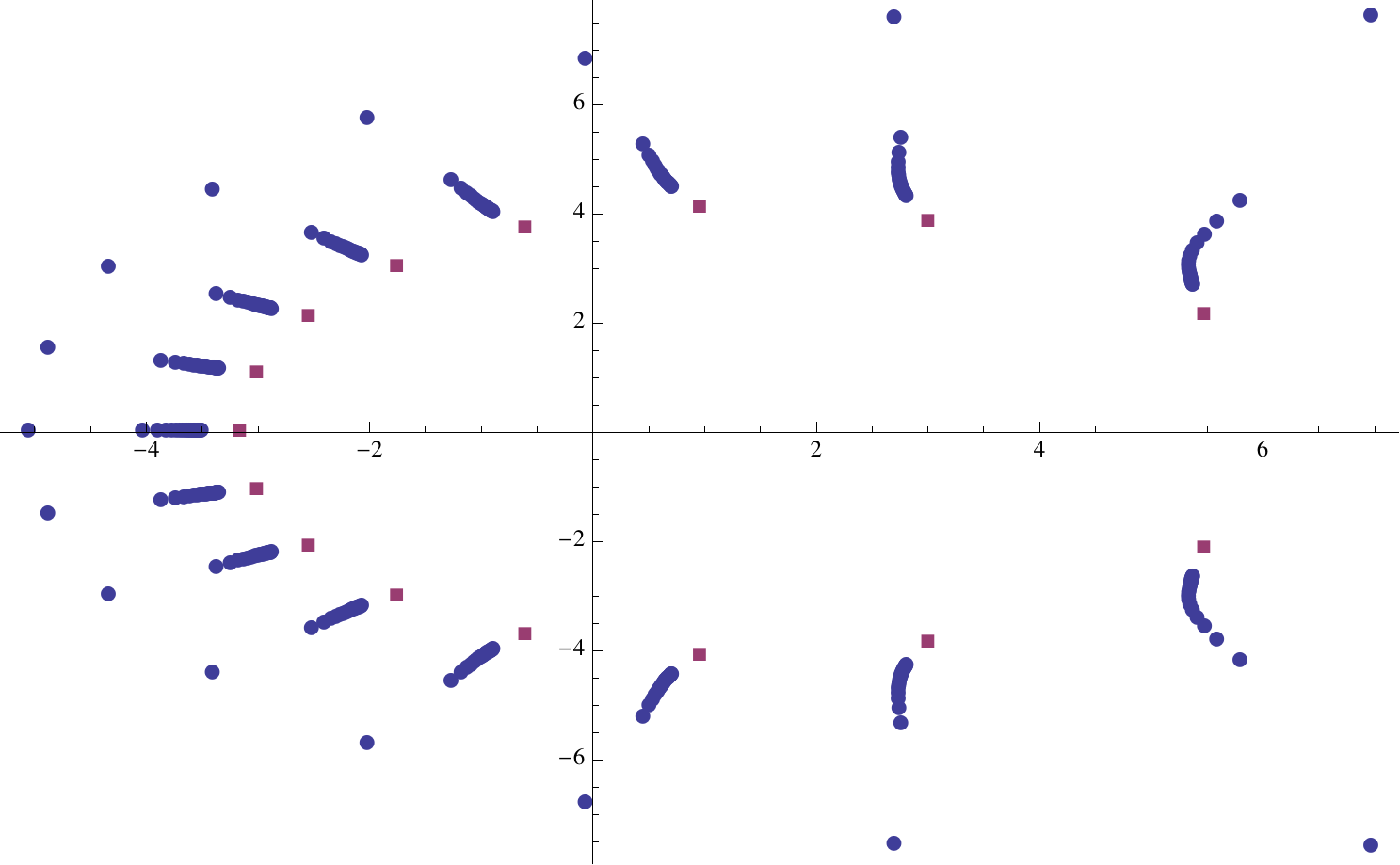}
\caption{Exceptional zeros of the polynomials $\XLagII{\al}{m}{m+j}(z)$ for
$m=15$, $\al=14.01$ and $1\leq j\leq 22$. The squares denote the zeros of
$\Lag{-\al-1}{m}(z)$ to which the zeros of $\XLagII{\al}{m}{m+j}(z)$
converge for $j\to\infty$. }
\end{figure}

\section{Exceptional Jacobi polynomials}
\subsection{Definitions and identities}
Let $m\geq 0$ be a fixed integer, and $\alpha,\beta$ real numbers.  Let
\begin{equation}
  T_{\al,\be}[y] = (1-z^2)y'' +
  (\beta-\alpha+(\alpha+\beta+2)z)y',
\end{equation}
denote the Jacobi differential operator.  The classical Jacobi
polynomial of degree $n$ can be defined as the polynomial solution
$y=\Jac{\alpha}{\beta}{n}(z)$ of the second order linear differential equation
\begin{equation}
  T_{\al,\be}[y] = -n(1+\al+\be+n) y,\quad  y(1) = \frac{(\al+1)_n}{n}.
\end{equation}
Next, define
\begin{align}
  T_{\alpha,\beta,m}[y] &= T_{\alpha,\beta}[y]+ (\alpha-\beta-m+1)m y
  \\ \nonumber &\qquad -(\log \Jac{-\al-1}{\be-1}{m})'\,
  \Big(\beta(1-z) y+(1-z^2)y'\Big),\\
  \label{eq:Aabmdef}
  A_{\alpha,\beta,m}[y] &= (1-z)\Jac{-\al}{\be}{m}\,
  y'+(m-\alpha)\Jac{-\al-1}{\be-1}{m}\, y,\\
  B_{\alpha,\beta,m}[y] &=
  \frac{(1+z)y'+(1+\beta)y}{\Jac{-\al}{\be}{m}}.
\end{align}

The following operator factorizations can be verified by the
application of elementary identities.
\begin{align}
  \label{eq:TabmAB}
  T_{\alpha,\beta,m} &= A_{\alpha+1,\beta-1,m} B_{\alpha+1,\beta-1,m}
  -(m-\alpha-1)(m+\beta) ,\\
  \label{eq:TabBA}
  T_{\alpha,\beta} &= B_{\alpha,\beta,m} A_{\alpha,\beta,m}
  -(m-\alpha)(m+\beta+1).
\end{align}
For $n\geq m$, we define the degree $n$ exceptional Jacobi polynomial
to be
\begin{align}
  \label{eq:hPdef}
  \XJac{\alpha}{\beta}{m}{n} &= \frac{(-1)^{m+1}}{\alpha+1+j}
  \,A_{\alpha+1,\beta-1,m}\left[
    \Jac{\alpha+1}{\beta-1}{j}\right],\qquad j=n-m\geq 0,\\
  &= \frac{(-1)^m}{\alpha+1+j} \left(\frac{1}{2}(1+\al+\be+j)(z-1)
    \Jac{-\al-1}{\be-1}{m}
    \Jac{\alpha+2}{\beta}{j-1}\right.\\ \nonumber
  &\qquad\qquad\qquad\left. +
    (\alpha+1-m)\,\Jac{-\al-2}{\be}{m}
    \Jac{\alpha+1}{\beta-1}{j}\right).
\end{align}
The exceptional polynomials and operator extend their classical
counterparts
\begin{align}
T_{\alpha,\beta,0}[y] &= T_{\alpha,\beta}[y],\\
\XJac{\alpha}{\beta}{0}{n} &= \Jac{\alpha}{\beta}{n}.
\end{align}

By construction, these polynomials satisfy several identities,
which we enumerate below.   The factorizations
\eqref{eq:TabmAB} \eqref{eq:TabBA} give the intertwining relations
\begin{gather}
  \label{eq:ATTA}
  A_{\alpha+1,\beta-1,m}T_{\alpha+1,\beta-1}    =
  T_{\alpha,\beta,m}A_{\alpha+1,\beta-1,m},\\
T_{\alpha+1,\beta-1}  B_{\alpha+1,\beta-1,m}  =
  B_{\alpha+1,\beta-1,m}T_{\alpha,\beta,m}.
\end{gather}
From the above relations we can derive the
eigenvalue equation for the $X_m$-Jacobi polynomials
\begin{equation}
  \label{eq:hPlambda}
  T_{\alpha,\beta,m} \left[\XJac{\alpha}{\beta}{m}{n}\right] =
  -(n-m)(1+\alpha+\beta+n-m)\hPud{\alpha,\beta}{m}{n},\quad n\geq m.\\
\end{equation}
The  factorization \eqref{eq:TabBA} implies the following identity
\begin{gather}
  \label{eq:BhP=P}
  (-1)^m(\alpha+1+j)(\beta \hPud{\alpha,\beta}{m}{m+j}+(z+1)
  \hPudp{\alpha,\beta}{m}{m+j})   =\\ \nonumber
  \qquad\qquad (\alpha+1-m+j) (\beta+m+j) \Pud{-\alpha-1,\beta-1}{m}
  \Pud{\alpha+1,\beta-1}{j}\,.
\end{gather}
It will be useful to express $\hP$ in a way that is symmetric in the dimension $j$
and the codimension $m$. Namely,
\begin{align}
  \nonumber
    & (-1)^m (\alpha+j)\hPud{\alpha-1,\beta+1}{m}{m+j} \\
    \label{eq:hPrel1}
  &\qquad = (\alpha-m)
  \Jac{-\alpha-1}{\beta+1}{m} \Jac{\alpha}{\beta}{j} + (z-1)
  \Jac{-\alpha}{\beta}{m} \Jacp{\alpha}{\beta}{j}\\
    \label{eq:hPrel2}
  &\qquad = (\alpha+j) \Pud{\alpha-1,\beta+1}{j}
  \Pud{-\alpha,\beta}{m} -
  (z-1)  \Jac{\alpha}{\beta}{j} \Jacp{-\alpha}{\beta}{m}\,.
\end{align}
The first equation is just a restatement of the definition
\eqref{eq:hPdef}, while the second identity  follows from the
classical relation
\begin{align}
  \label{eq:z-1P'}
  (z-1) \Jacp{\alpha}{\beta}{j} &=
  \alpha\Pud{\alpha,\beta}{j} -(\alpha+j) \Pud{\alpha-1,\beta+1}{j}\,.
\end{align}
At the endpoints of the interval of orthogonality we have the following classical identities
\begin{align}
  \label{eq:Pab-1}
   \Pud{\al,\be}{n}(-1) &= (-1)^n \frac{(\be+1)_n}{n!},\\
  \label{eq:Pab+1}
   \Pud{\al,\be}{n}(1) &=  \frac{(\al+1)_n}{n!},
\end{align}
which in the case of exceptional Jacobi polynomials yield the following
generalizations
\begin{align}
   \label{eq:hP+1}
   \hPud{\alpha,\beta}{m}{n}(1) &=
   \binom{\alpha+n-m}{n} \binom{n}{m} ,\quad n\geq m,\\
   & =   \frac{(\alpha+1-m)_{m+j}}{m!\,j!} ,\quad j=n- m,  \\
    \label{eq:hP-1}
    \hPud{\alpha,\beta}{m}{m+j}(-1) &=
    (-1)^j\frac{(\beta+j+m)(1+\alpha-m+j)}{(1+\alpha+j)}
    \frac{(\beta+1)_{m-1}(\beta)_{j}}{m!\,j!}.
\end{align}

Define the 1st-order operators
\begin{equation}
    \label{eq:hAabmdef}
  \hA_{\alpha,\beta,m}[y]
  =\frac{\Jac{-\al-2}{\be}{m}}{\Jac{-\al-1}{\be-1}{m}} \left(y'
  -(\log\Jac{-\al-2}{\be}{m})' \, y\right)\,,
\end{equation}
\begin{equation}
  \hB_{\alpha,\beta,m}[y]
  =(1-z^2)\frac{\Jac{-\al-1}{\be-1}{m}}{\Jac{-\al-2}{\be}{m}} \left[
    y'- \left((\log\Jac{-\al-1}{\be-1}{m})'+
      \frac{\alpha+1}{1-z}-\frac{\beta+1}{1+z}\right) y \right]\,.
\end{equation}
The following ``shape-invariant'' factorizations relate
exceptional operators of the same codimension at different values of
the parameters $\alpha,\beta$
\begin{align}
  \hB_{\alpha,\beta,m} \hA_{\alpha,\beta,m} &=T_{\alpha,\beta,m} \\
  \hA_{\alpha,\beta,m} \hB_{\alpha,\beta,m} &=T_{\alpha+1,\beta+1,m}
+2+\alpha+\beta.
\end{align}
The corresponding intertwining relations, namely,
\begin{align}
  T_{\alpha+1,\beta+1,m} \hA_{\alpha,\beta,m} &=\hA_{\alpha,\beta,m}
T_{\alpha,\beta,m}, \\
  \hB_{\alpha,\beta,m} T_{\alpha+1,\beta+1,m} &= T_{\alpha,\beta,m}
\hB_{\alpha,\beta,m},
\end{align}
give rise to the lowering and raising relations for the exceptional Jacobi
polynomials
\begin{gather}
  \label{eq:hPlower}
  \left( \frac{ \XJac{\alpha}{\beta}{m}{n}}{\Jac{-\al-2}{\be}{m}}\right)'
  = \frac{1}{2}(n-m+\alpha+\beta+1)
  \left(\frac{\Jac{-\al-1}{\be-1}{m}}{\Jac{-\al-2}{\be}{m}}\right)^2\,
  \frac{\XJac{\alpha+1}{\beta+1}{m}{n-1}}{\Jac{-\al-1}{\be-1}{m}},\quad n\geq m+1,\\
  \label{eq:hPraise}
  (1-z)^{-\alpha}(1+z)^{-\beta}\left((1-z)^{\alpha+1}
      (1+z)^{\beta+1}\frac
      {\XJac{\alpha+1}{\beta+1}{m}{n}}{\Jac{-\al-1}{\be-1}{m}}\right)' = \\ \nonumber
  \qquad \qquad -2(n-m+1)
  \left(\frac{\Jac{-\al-2}{\be}{m}}{\Jac{-\al-1}{\be-1}{m}}
  \right)^2\frac{\XJac{\alpha}{\beta}{m}{n+1}}{\Jac{-\al-2}{\be}{m}}
  ,\quad n\geq m.
\end{gather}
As usual, we denote by $f'(z)$ the derivative of $f$ with respect to
the $z$ variable.

\subsection{Orthogonality}
The exceptional Jacobi polynomials are formally orthogonal with
respect to $\hW_{\alpha,\beta,m}(z) dz,\; -1\leq z\leq 1$, where
\begin{equation}
  \label{eq:hWjacdef}
  \hW_{\alpha,\beta,m}(z) = \frac{(1-z)^\alpha
    (1+z)^\beta}{\left[\Jac{-\al-1}{\be-1}{m}(z)\right]^2}.
\end{equation}
In order to have orthogonality in the $L^2$ sense, additional conditions need to
be imposed on the parameters
$\alpha, \beta$, and $m$.  The condition $\alpha,\beta>-1$ is necessary for the
measure \eqref{eq:hWjacdef} to have
finite moments of all orders.  Another requirement is that the denominator
$\Jac{-\al-1}{\be-1}{m}$ does not vanish  for $z\in(-1,1)$, which imposes extra
conditions on $\alpha, \beta$, and $m$.

An analysis of the zeros of classical Jacobi polynomials can be found in
Szeg\H{o}'s book \cite[Chapter 6.72]{Sz}.  First, let us recall that
$\Pud{\alpha,\beta}{n}(z)$ has a zero of multiplicity $k$ at $z=1$ if
$\alpha=-k,\;
k=1,\ldots, m$, and a zero of multiplicity $j$ at $z=-1$ if $\beta=-j,\;
j=1,\ldots, m$.

We also mention the degenerate cases where
$$\deg \Pud{\alpha,\beta}{n} = k  \quad \text{ when } \quad
n+\alpha+\beta+1=-k,\qquad k=0,1,\ldots, n-1.$$
For such  parameter values the $n$th Jacobi polynomial has degree $<n$.  In these degenerate cases (where $\al+\be$
is a negative integer), we have
\begin{equation}
  \label{eq:Pdegen}
  \binom{\al+m}{m} \Pud{\al,\be}{n}(z) = \binom{\al+n}{n}
  \Pud{\al,\be}{m}(z),\quad \al+\be=-1-m-n.
\end{equation}

Since $\beta>-1$, the denominator has a zero at $z=-1$ if and only if
$\beta=0$.  However, the latter condition gives a weight with an
overall factor of $(1+z)^{-1}$, which would violate the assumption that it has
finite moments of all orders. Therefore we must impose $\beta\neq 0$. The
condition that
$\Pud{-\alpha-1,\beta-1}{m}(z)\neq0$ for $z\in(-1,1)$ is satisfied in
exactly two cases
\begin{itemize}
\item[(A)]Both $\beta$ and $\alpha+1-m$ $\in (-1,0)$.
\item[(B)]Both $\beta$ and $\alpha+1-m$ $\in (0,\infty)$.
\end{itemize}
For (A) we have $ m-2<\alpha< m-1$, while for
(B) we have $\alpha>m-1$.  Therefore, in both cases
$\Pud{-1-\alpha,\beta-1}{m}(1) \neq 0$.
According to identity \eqref{eq:Pdegen}, we also
require
\begin{equation}
  \label{eq:al1mnotin}
  \alpha+1-m-\beta\notin\{0,1,\ldots, m-1\}.
\end{equation}
If this condition is violated, then $\deg
\Pud{-\alpha-1,\beta-1}{m}(z)<m$ and, therefore, the codimension (see
below for discussion) is $\alpha+1-m-\beta$, rather than $m$.  We
therefore append condition \eqref{eq:al1mnotin} to the assumptions in
(A) and (B), to complete the following

\begin{prop}
  \label{prop:albe1}
  Suppose that $m\geq 1$.  The measure $\hW_{\alpha,\beta,m} dz,\;
  z\in(-1,1)$ is positive definite with finite moments of all orders if and only
if $\alpha,\beta,m$ satisfy one of the following  conditions
\begin{itemize}
\item[(A)] $\beta,\alpha+1-m\in (-1,0)$.
\item[(B)] $\beta,\alpha+1-m\in (0,\infty)$.
\end{itemize}
In order to ensure that  $\deg \Jac{-\alpha-1}{ \beta-1}{m} = m$  it is
  also necessary to require that $\alpha+1-m-\beta\notin\{0,1,\ldots,
  m-1\}$.
\end{prop}
\noindent

\subsection{Exceptional Flag}
Let us define the following codimension $m$ polynomial flag $\{ U_{\alpha,\beta,m,j} \}_{j=1}^\infty$ where
\[ U_{\alpha,\beta,m,j} = \{ f(z) \in \cP_{m+j-1}(z) :
\Jac{-\al-1}{\be-1}{m} | (1+z) f' + \beta f \}.\] At the level of
flags, the factorizations \eqref{eq:TabmAB} \eqref{eq:TabBA}
correspond to the linear isomorphisms
\[ A_{\alpha+1,\beta-1,m}: \cP_{j-1} \to U_{\alpha,\beta,m,j},\quad
B_{\alpha+1,\beta-1,m} : U_{\alpha,\beta,m,j} \to \cP_{j-1},\quad
j=1,2,\ldots. \] Thus, the exceptional polynomials
$\hPud{\alpha,\beta}{m}{m+j}$ give a basis of the flag
$\{ U_{\alpha,\beta,m,j} \}_{j=1}^\infty$

\begin{prop}
  Suppose that $\alpha,\beta,m$ satisfy either condition (A) or condition
   (B). Then $\XJac{\alpha}{\beta}{m}{m+j}(z)$ is the
  unique polynomial in $U_{\alpha,\beta,m,j}$,  orthogonal to
  $U_{\alpha,\beta,m,j-1}$ with respect to $\hW_{\alpha,\beta,m}\,dz, \;
  z\in (-1,1)$  that satisfies the normalizing condition
  \eqref{eq:hP+1}.
\end{prop}


Once we have analyzed the underlying factorizations that give rise to $X_m$-Jacobi polynomials, and the conditions on the parameters that ensure their $L^2$-orthogonality, we can now turn our attention to describing some properties of their zeros.

\subsection{Zeros of exceptional Jacobi polynomials}
Let us refer to the real zeros of $\hPud{\alpha,\beta}{m}{m+j}(z)$ in
$z\in (-1,1)$ as the \emph{regular zeros}. All other zeros, whether in
$(-\infty,-1)\cup (1,\infty)$, or complex, will be said to be
\emph{exceptional zeros}.



\begin{prop}
    Suppose that $\alpha,\beta,m$ obey either condition (A) or
condition (B).  Then the regular zeros of
  $\hPud{\alpha,\beta}{m}{m+j}(z)$ are simple.
\end{prop}
\begin{proof}
  This follows from \eqref{eq:hPdef} \eqref{eq:BhP=P} and the
  simplicity of the zeros of classical Jacobi polynomials.
\end{proof}

\begin{prop}
  Suppose that $\alpha,\beta,m$ obey either condition (A) or
condition (B).  Then $\hPud{\alpha,\beta}{m}{m+j}$ has exactly $j$
  regular zeros and $m$ exceptional zeros.
\end{prop}
\begin{proof}
  We begin by showing that $\XJac{\alpha}{\beta}{m}{m+j}$ has at least $j$
  regular zeros  by induction on $j$.  The case $j=0$ is
  trivial.  Suppose now that the proposition has been established for
  $j\geq 0$.  Note that $\alpha+1,\beta+1,m$ always belong to class B by hypothesis. We observe also \cite[Chapter 6.72]{Sz} that $\Jac{-\al-1}{\be-1}{m}(z)$
  and $\Jac{-\al-2}{\be}{m}(z)$ have no zeros in $z\in [-1,1]$.
  Let $\zeta_1,\ldots, \zeta_i,\; i\geq j$,
  be the regular zeros of $\XJac{\alpha+1}{\beta+1}{m}{m+j}$.
  By \eqref{eq:hPraise} and  Rolle's theorem,
  $\XJac{\alpha}{\beta}{m}{m+j+1}(z)$ has at least one zero in each of the intervals
  $(\zeta_1,\zeta_2), (\zeta_2,\zeta_3),\ldots,
  (\zeta_{i-1},\zeta_i)$.   There will also be a zero in
  $(-1,\zeta_1)$ and a zero in $(\zeta_i,1)$, for a total of at least
  $i+1$ zeros.

  We conclude by showing that $\XJac{\alpha}{\beta}{m}{m+j}$ has at most
  $j$ regular zeros.  The proof is again by induction on $j$.  Observe
  that
  \[ \XJac{\alpha}{\beta}{m}{m} = (-1)^m \left(\frac{\alpha+1-m}{\alpha+1+j}\right)
  \Jac{-\al-2}{\be}{m},\] and the latter has no zeros in
  $[-1,1]$. Relation \eqref{eq:hPlower} shows that between two regular
  zeros of $\XJac{\alpha}{\beta}{m}{m+j}$ there is at least one zero of
  $\XJac{\alpha+1}{\beta+1}{m}{m+j-1}$.  Hence, if we assume that the
  latter has at most $j-1$ regular zeros, then the former has at most
  $j$ regular zeros.
\end{proof}

\subsection{Asymptotic behaviour of the zeros}
Our next goal is to derive a representation for the $X_m$-Jacobi polynomials that
is amenable  to asymptotic analysis.
\begin{prop}
  The following identity holds:
  \begin{align}
    \label{eq:xjacaltrep}
    &(-1)^m \XJac{\alpha}{\beta}{m}{m+j} = \\ \nonumber &\qquad
    \frac{\alpha-\beta-m+1}{1+\alpha+j}
    \Pud{-\alpha,\beta}{m-1}\left(\frac{j}{\alpha+\beta+2j}
      \Pud{\alpha,\beta}{j}-\frac{\alpha+j}{\alpha+\beta+2j}
      \Pud{\alpha,\beta}{j-1} \right) + \\ \nonumber &\qquad
    \left(\frac{1+\alpha-m}{1+\alpha+j} \Pud{-2-\alpha,\beta}{m} +
      \frac{j}{1+\alpha+j} \Pud{-\alpha-1,\beta-1}{m} \right)
    \Pud{\alpha,\beta}{j}\,.
  \end{align}
\end{prop}
\begin{proof}
  We begin with \eqref{eq:hPrel2} and apply \eqref{eq:z-1P'} as well
  as the following classical identities:
  \begin{align*}
    \Pud{1+\alpha,\beta-1}{j} &= \Pud{\alpha+1,\beta}{j-1} +
    \Pud{\alpha,\beta}{j}\\
    (\alpha+\beta+2j)(z-1) \Pud{\alpha+1,\beta}{j-1} &= 2j
    \Pud{\alpha,\beta}{j} - 2(\alpha+j) \Pud{\alpha,\beta}{j-1}
  \end{align*}
\end{proof}

\begin{prop}
  Suppose that $\alpha,\beta>-1$. As $j\to \infty$ we have the
  following asymptotic behaviour for $z$ in compact sets of $\mathbb
C\backslash[-1,1]$
  \begin{equation}
    \label{eq:hPasym}
    \hPud{\alpha,\beta}{m}{m+j} -
    (-1)^m\Jac{-\al-1}{\be-1}{m} \Jac{\al}{\be}{j}\rightrightarrows 0,\qquad j\to\infty.
  \end{equation}
\end{prop}
\begin{proof}
  We make use of the following well known ratio asymptotics formula
  for classical Jacobi polynomials:
  \begin{equation}
    \frac{\Pud{\alpha,\beta}{j-1}(z)}{\Pud{\alpha,\beta}{j}(z)}
    \rightrightarrows \frac{1}{z+\sqrt{z^2-1}},\quad z\notin [-1,1].
  \end{equation}
  The conclusion now follows directly from \eqref{eq:xjacaltrep}.
\end{proof}

As a straightforward consequence, the following corollary describes the
asymptotic behaviour of the
zeros of exceptional Jacobi polynomials.
\begin{cor}
 The $j$ regular zeros of $ \hPud{\alpha,\beta}{m}{m+j}$ approach the zeros of
the classical Jacobi polynomial  $\Pud{\alpha,\beta}{j}$ as $j\to\infty$, while
the $m$ exceptional zeros of $\hPud{\alpha,\beta}{m}{m+j}$ approach the zeros of
$\Pud{-\alpha-1,\beta-1}{m}$.
\end{cor}

%

The Heine-Mehler formula for the classical Jacobi polynomials states
\begin{equation}
  \label{eq:hmjacclass}
  n^{-\alpha} \Pud{\alpha,\beta}{n}(\cos(z/n)) \rightrightarrows
  (z/2)^{-\alpha}J_\alpha(z).
\end{equation}
The $X_m$-Jacobi polynomials satisfy a generalized Heine-Mehler
formula, given by the following proposition.
\begin{prop}
  When $n\to \infty$, we get
  \begin{equation}
    \label{eq:hmxjac}
    n^{-\alpha} \hPud{\alpha,\beta}{m}{n}(\cos(z/n)) \rightrightarrows
    (-1)^m(z/2)^{-\alpha}
    \binom{m-1-\alpha}{m} J_\alpha(z),\quad n\geq m.
  \end{equation}
\end{prop}
\begin{proof}
 Taking the limit $j\to\infty$ in  \eqref{eq:xjacaltrep} and using the
classical Heine-Mehler formula \eqref{eq:hmjacclass} leads to the desired
result,
keeping in mind also that
  \[ \Pud{-\alpha-1,\beta-1}{m}(1) = \binom{m-1-\alpha}{m}.\]
\end{proof}

\section{Summary and Open problems}

We have provided suitable representations of exceptional polynomials in terms
of their classical counterparts by exploiting the isospectral Darboux
transformations that connect them. These representations allow to derive
Heine-Mehler type formulas for the exceptional Jacobi and Laguerre polynomials,
which describe the asymptotic behaviour of their \textit{regular zeros} (those
lying in the interval of orthogonality). The behaviour of the regular zeros of
exceptional polynomials follows the same Bessel asymptotics as the zeros of
their classical counterparts. We have also proved interlacing between
the zeros of exceptional and classical polynomials, while the zeros of
consecutive exceptional polynomials also interlace according to their
Sturm-Liouville character. As for the \textit{exceptional zeros} (those lying
outside the interval of orthogonality) we have established their number and
location and we have proved that for fixed codimension $m$ and large degree $n$
they approach the zeros of a classical polynomial. We have performed a careful
analysis of the admissible ranges of the parameters that ensure a well defined
Sturm-Liouville problem. We have also given raising and lowering relations for the exceptional polynomials. These relations correspond to a \textit{shape-invariant} factorization, i.e. a Darboux transformation that falls within the same class of operators, with a shift in the parameters (see for instance \eqref{sifact1}--\eqref{sifact2}), and they imply that the associated potentials in quantum mechanics will be exactly solvable and shape invariant.

It was recently noticed that more families of exceptional orthogonal polynomials can be constructed through
 \emph{multi-step} Darboux or Darboux-Crum transformations \cite{GKM12a}, an idea that has been further developed
 in  \cite{Grandati1,Quesne3,SO5}.  In this work we have analyzed the zeros of exceptional orthogonal polynomials
that can be obtained from the classical ones by a 1-step Darboux transformation.
These polynomials can be written as a first order linear
differential operator acting on their classical counterparts and the
exceptional weight is a classical weight divided by the square of a classical
polynomial with zeros outside the interval of orthogonality. An open problem is
to extend this analysis to multi-step exceptional families, where exceptional
polynomials are obtained by the action of an $m$-th order differential operator.

We believe that all exceptional orthogonal polynomials can be obtained
from a classical system by a multi-step Darboux transformation, \cite{GKM12c},
and the exceptional weight for these systems will have in its denominator the
Wronskian of all the factorizing functions, which are essentially classical
orthogonal polynomials. The characterization of all such Wronskians whose zeros
lie outside the interval of orthogonality becomes then a crucial question.

It is trivial to know the location of the zeros of a classical
polynomial, and therefore to constrain its parameters so that they fall
outside the interval of orthogonality for the exceptional weight to be regular. The question becomes much more involved when dealing with a Wronskian
of classical polynomials as it happens in the multi-step case. However, this question must be addressed in order to select those
multistep weights that are non-singular.

The position of the zeros for Wronskians of consecutive Hermite polynomials have
been investigated numerically by Clarkson \cite{clarkson1} since these functions
appear as rational solutions to nonlinear differential equations of Painlev\'e
type, \cite{clarkson2}. A further numerical analysis together with some
conjectures in a more general case have been recently put forward by Felder et
al.\cite{felder}, in connection with the theory of
monodromy free potentials. We stress that a Wronskian of classical polynomials
might have no real zeros even if the polynomials themselves do. The Adler-Krein
theorem \cite{adler,krein} provides a useful criterion to identify these cases,
and it is actually a much more general result for eigenfunctions of a
Schr\"odinger operator, not just polynomials. A generalization of this result is
being carried out by Grandati, who is extending the analysis to factorizing
functions of isospectral Darboux transformations \cite{Grandati1}, as opposed to
the Adler-Krein case which refers only to state-deleting Darboux
transformations, for which the factorizing functions are true $L^2$
eigenfunctions. The most general problem of Wronskians that involve factorizing
functions of mixed type remains unsolved.

\vskip 0.6cm
\paragraph{\textbf{Acknowledgements}}
\thanks{

The research of the first author (DGU) has been
supported by Direcci\'on General de
Investigaci\'on, Ministerio de Ciencia e Innovaci\'on of Spain, under
grant MTM2009-06973. The work of the second author
(FM) has been supported by Direcci\'on General de
Investigaci\'on, Ministerio de Ciencia e Innovaci\'on of Spain, grant
MTM2009-12740-C03-01. The research of the thirds author (RM) was supported in
part by NSERC grant RGPIN-228057-2009.
}

\end{document}